\documentclass[a4paper]{article}
\usepackage{amsmath,amsthm,amssymb,amscd,ascmac, amsfonts}
\usepackage{mathrsfs}
\usepackage{stmaryrd}
\usepackage{braket}
\usepackage{color}
\usepackage{accents}
\usepackage{url}
\usepackage[english]{babel}
\usepackage[margin=20truemm]{geometry}
\usepackage[title,titletoc]{appendix}
\allowdisplaybreaks[4]

\newcommand{\Zz}{\mathbb{Z}}
\newcommand{\Cz}{\mathbb{C}}

\newcommand{\sgn}{{\rm sgn}}

\numberwithin{equation}{section}
\theoremstyle{plain}
\newtheorem{thm}{Theorem}

\newtheorem{prop}{Proposition}
\newtheorem{lem}{Lemma}
\newtheorem{coro}{Corollary}

\newtheorem{rmk}{Remark}

\begin{document}
\title{Factorized $q$-difference equations and their fundamental solutions}
\author{S. Tsuchimi}
\date{}

\maketitle

\begin{abstract}
We construct fundamental solutions of a class of factorized $q$-difference equations in terms of Jacobi theta functions and Zwegers' $\mu$-functions. 
To obtain a formal series solution $f_0$ for a factorized  $q$-difference equation, we present a $q$-analogue of the fundamental relation (often called the Wronski-type relation) between the elementary symmetric polynomials and the complete homogeneous symmetric polynomials. 
We also provide some formulas for the image of the $q$-Borel summation of the formal series solution $f_0$. 
As an application, we give another proof of the expressions due to S. Y. Kang and S. P. Zwegers of the universal mock theta functions and higher level Appell functions in terms of $\mu$-functions. 
\end{abstract}

\section{Introduction}

Throughout this paper, let $i:=\sqrt{-1}$ be the imaginary unit, ${\rm Im}(\tau)>0$, $q:=e^{2\pi i\tau}$, and we define the $q$-shifted factorials, the Jacobi theta function and Zwegers' $\mu$-function as follows:
\begin{align}
(x)_{\infty }
   &=(x;q)_{\infty }:=\prod_{j=0}^{\infty }(1-xq^{j}), \quad 
(x)_{\nu}
   =(x;q)_{\nu}:=\frac{(x;q)_{\infty }}{(q^{\nu}x;q)_{\infty }} \quad (\nu \in \Cz),\\
(x_1,\ldots,x_r)_\nu
  &=(x_1,\ldots,x_r;q):=(x_1)_\nu\cdots(x_r)_\nu, \\
\theta_q(x)
   &:=\sum_{n\in\Zz}x^nq^\frac{n(n-1)}{2}=(q)_\infty(-x)_\infty(-q/x;q)_\infty, \\
   \mu(x,y;q)=\mu(x,y)&:=iq^{-\frac{1}{8}}\frac{\sqrt{xy}}{\theta_q(-y)}\sum_{n\in\Zz}\frac{(-1)^ny^nq^\frac{n(n+1)}{2}}{1-xq^n}\quad (x,y\not\in q^{\Zz}). 
\end{align}
In particular, the theta function and $\mu$-function are appropriately written in the following form
\begin{align}
   \vartheta(u)
   &:=\sum_{n\in\Zz+\frac{1}{2}}e^{2\pi inu}q^\frac{n^2}{2}=-ie^{-\pi iu}q^\frac{1}{8}\theta_q\left(-e^{2\pi iu}\right), \\
   \label{mu}
 \mu(u,v;\tau)=\mu(u,v)&:=\frac{e^{\pi iu}}{\vartheta(v)}\sum_{n\in\Zz}\frac{(-1)^ne^{2\pi inv}q^\frac{n(n+1)}{2}}{1-e^{2\pi iu}q^n}\quad (u,v\not \in\{m+n\tau;\ m,n\in\Zz\}).
\end{align} 
Regarding the $\mu$-function, note that we use the same symbol $\mu$ for both multiplicative and additive variables (an abuse of notation). 
Further, we define
\begin{align*}
\mathfrak{M}_{M+1}(x,a_0,\ldots, a_M)
  &:=
  -\frac{1}{\theta_q(-x)}
   \sum_{n\in\Zz}\left(-\frac{1}{x}\right)^nq^\frac{n(n+1)}{2}\prod_{j=0}^M\frac{1}{1-(q^n/a_j)}, 
\end{align*} 
for $x, a_j\not\in q^{\Zz}$ $(j=0,\ldots, M)$. 
For simplicity, we write $\mathfrak{M}_{M+1}(x,a_0,\ldots, a_M)=\mathfrak{M}(x,a_0,\ldots, a_M)$. 

Mock theta functions first appeared in Ramanujan's last letter to Hardy in 1920. 
Ramanujan presented several identities for the mock theta functions. 
These identities were later shown to imply that all mock theta functions are expressed as specializations of the universal mock theta functions \cite{BFOR}
\begin{align}
g_2(x;q)&:=\sum_{n=0}^\infty\frac{(-q)_nq^\frac{n(n+1)}{2}}{(x,x^{-1}q)_{n+1}}=\frac{(-q)_\infty}{(q)_\infty}\sum_{n\in\Zz}\frac{(-1)^nq^{n(n+1)}}{1-xq^n}, \\
g_3(x;q)&:=\sum_{n=0}^\infty\frac{q^{n(n+1)}}{(x,x^{-1}q)_{n+1}}=\frac{1}{(q)_\infty}\sum_{n\in\Zz}\frac{(-1)^nq^\frac{3n(n+1)}{2}}{1-xq^n}.
\end{align}

Thus, the universal mock theta functions serve as master functions for Ramanujan's mock theta functions. 

In $2009$, S. Y. Kang showed the following expressions of the universal mock theta functions by the $\mu$-function \cite[Theorem 1.1 and Theorem 3.1]{Ka}. 
\label{K}
\begin{align}
\label{K1}
g_2(x;q)&=-iq^{-\frac{1}{4}}\mu(x^2,q;q^2)+\frac{(q^2;q^2)_\infty^4}{(q;q)_\infty^2\theta_{q^2}(-x^2)},\\
\label{K2}
g_3(x;q)&=-ix^{-\frac{1}{2}}q^{-\frac{1}{8}}\mu(x^3,q;q^3)-ix^\frac{1}{2}q^{-\frac{5}{8}}\mu(x^3,q^2;q^3)+\frac{(q^3;q^3)_\infty^3}{(q;q)_\infty\theta_{q^3}(-x^3)}.
\end{align}

More generally, the universal mock theta functions are special cases of the higher level Appell function
\begin{align}
\label{eq: higher}
A_m(x,y):=x^\frac{m}{2}\sum_{n\in\Zz}\frac{(-1)^{mn}y^nq^\frac{mn(n+1)}{2}}{1-xq^n}
\end{align}
introduced by Kac-Wakimoto \cite{KW}. From a simple calculation, we have
\begin{align*}
A_1(x,y)=-iq^\frac{1}{8}\frac{\theta_q(-y)}{\sqrt{y}}\mu(x,y),\quad g_2(x;q)=\frac{(-q)_\infty}{(q)_\infty}x^{-1}A_2(x,-1),\quad g_3(x;q)=\frac{1}{(q)_\infty}x^{-\frac{3}{2}}A_3(x,1).
\end{align*}
Zwegers showed the following expression for $A_m(x,y)$ in terms of the $\mu$-function \cite[Lemma 2]{Zw2}: 
\begin{align}
\label{Z}
A_m(x,(-1)^{m-1}y)=\sum_{k=0}^{m-1}-iq^\frac{m}{8}x^k\frac{\theta_{q^m}(-yq^k)}{\sqrt{yq^k}}\mu(x^m,yq^k;q^m), 
\end{align}
and \eqref{Z} specializes to \eqref{K1} and \eqref{K2}. 

The $\mu$-function has many rich properties.  As a particularly important property, the $\mu$-function satisfies  a transformation law like Jacobi forms by adding an appropriate non-holomorphic function to the $\mu$-function. 
Namely, we define 
\begin{align}
E(z)
  &:=
  2\int_0^z e^{-\pi iu^2}\ du, 
\\
R(u;\tau)
\label{eq: defi R}
  &:=
  \sum_{n\in\Zz+\frac{1}{2}}\left(\sgn\left(n\right)-E\left(\left(\frac{{\rm Im}(u)}{{\rm Im}(\tau)}+n\right)\sqrt{2{\rm Im}(\tau)}\right)\right)(-1)^ne^{-2\pi inu}q^{-\frac{n^2}{2}}, 
  \\
\widetilde{\mu}(u,v;\tau)
\label{eq: comp mu}
  &:=
  \mu(u,v;\tau)+\frac{i}{2}R(u-v;\tau), 
\end{align}
then the function $\widetilde{\mu}(u,v;\tau)$ satisfies the following modular transformations \cite[Theorem 1.11]{Zw1}: 
\begin{align*}
\widetilde{\mu}(u,v;\tau)=e^\frac{\pi i}{4}\widetilde{\mu}(u,v;\tau+1), \quad \widetilde{\mu}(u,v;\tau)=-\frac{e^\frac{\pi i(u-v)^2}{\tau}}{\sqrt{-i\tau}}\widetilde{\mu}\left(\frac{u}{\tau},\frac{v}{\tau};-\frac{1}{\tau}\right). 
\end{align*}
This result was a starting point in the study of mock modular forms. 

There have been recent studies constructing the Zwegers' $\mu$-function using the $q$-Borel and $q$-Laplace transformations (\cite{GW} and \cite{ST}). 
The $q$-Borel transformation $\mathcal{B}_q$ and the $q$-Laplace transformation $\mathcal{L}_q$ are one of the linear operators that construct convergent solutions from divergent series solutions of $q$-difference equations introduced by \cite{RSZ}. 
For a formal series $g(x)=\sum A_n x^n$, the definitions of $\mathcal{B}_q$ and $\mathcal{L}_q$ are as follows: 
\begin{align*}
\mathcal{B}_q(g)(\xi):=\sum A_nq^\frac{n(n-1)}{2} \xi^n,\quad \mathcal{L}_q(g)(x,\lambda):=\sum_{n\in\Zz}\frac{g(\lambda q^n)}{\theta_q(\lambda q^n/x)}. 
\end{align*}
From a direct calculation, for $m,n\in\Zz$, we have
\begin{align}
\label{eq: q-Laplace}
\mathcal{B}_q(1)(\xi)&=\mathcal{L}_q(1)(x,\lambda)=\mathcal{L}_q\circ\mathcal{B}_q(1)(x,\lambda)=1, 
  \\
\label{eq: q-Borel trans}
\mathcal{B}_q(x^mT_x^ng)(\xi)&=q^\frac{m(m-1)}{2}\xi^mT_\xi^{m+n}\mathcal{B}_q(g)(\xi), 
 \\
\label{eq: q-Laplace trans}
\mathcal{L}_q(x^mT_x^ng)(x,\lambda)&=q^{-\frac{m(m-1)}{2}}x^mT_x^{n-m}\mathcal{L}_q(g)(x,\lambda), 
  \\
\label{operator methods}
\mathcal{L}_q\circ\mathcal{B}_q(x^mT_x^ng)(x,\lambda)&=x^mT_x^n\mathcal{L}_q\circ\mathcal{B}_q(g)(x,\lambda), 
\end{align}
where $T_xg(x):=g(xq)$. 
Therefore, for a divergent series solution $g(x)$ of a $q$-difference equation, if the image of the composition $\mathcal{L}_q\circ\mathcal{B}_q(g)(x,\lambda)$ converges, we obtain a convergent solution of the original $q$-difference equation from the above operator methods \eqref{operator methods}. 

In \cite{ST}, the authors studied the following $q$-Hermite-Weber equation: 
\begin{align}
\label{H}
[T_x^2-(1-xq)\sqrt{a}T_x-xq]f(x)=0.
\end{align}
From a simple calculation, the equation \eqref{H} has the following divergent formal solution around $x=0$:
\begin{align}
\label{FS}
x^\frac{\alpha}{2}\tilde{f}_0(x):=x^\frac{\alpha}{2}\sum_{n=0}^\infty\frac{(a)_n}{(q)_n}q^{-\frac{n(n-1)}{2}}\left(-\frac{x}{a}\right)^n\quad q^\alpha:=a, 
\end{align}
and the authors constructed the following special solution of \eqref{H} that transforms $x$ into $x/y$ by applying the composite $\mathcal{L}_q\circ\mathcal{B}_q$ to this formal solution \eqref{FS}: 
\begin{align}
\mu(x,y;a;q)&:=-iq^{-\frac{1}{8}}\frac{(x/y)^\frac{\alpha}{2}}{\theta_q(-y)}\sum_{n\in\Zz}(-y)^nq^\frac{n(n-1)}{2}\frac{(xq^{n})_\infty}{(xq^{n}/a)_\infty}
  \nonumber\\
  \label{def: gmu}
  & =
  (x/y)^\frac{\alpha}{2}\mathcal{L}_q\circ\mathcal{B}_q(\tilde{f}_0)(x/y,-x). 
\end{align}
In particular, when $a=q$, the equation \eqref{H} is factorized and has the $\mu$-function as a special solution, and $\mu(x,y;a)$ coincides with the $\mu$-function. 
Namely, 
\begin{align}
\label{eq: mu eq}
[T_x^2-(1-xq/y)\sqrt{q}T_x-xq/y]\mu(x,y)=\left[T_x-q^\frac{1}{2}\right]\left[T_x+\frac{x}{y}q^\frac{1}{2}\right]\mu(x,y)=0
\end{align}
and $\mu(x,y;q;q)=\mu(x,y)$.

In this paper, we focus on the fact that the $\mu$-function, which is a building block of mock theta functions, satisfies the factorized $q$-difference equation \eqref{eq: mu eq}. 
The factorization structure naturally appears in several special $q$-difference equations, including the $q$-Hermite-Weber equation. 
It is therefore natural to investigate a general class of factorized $q$-difference equations and describe their fundamental solutions. 
Then, we analyze the following $q$-difference equation: 
\begin{align}
\label{linear eq}
  &\left[\prod_{k=0}^N(T_x-b_k)\right]\left[\prod_{M\geq j\geq0}^{\longrightarrow}(T_x+a_jxq^j)\right]f(x)=0, 
\end{align}
where $a_0,\ldots,a_M$ and $b_0,\ldots,b_N$ are distinct complex parameters respectively and
\begin{align}
\label{linear eq1}
  &\left[\prod_{M\geq j\geq0}^{\longrightarrow}(T_x+a_jxq^j)\right]f(x)
  :=
  (T_x+a_Mxq^M)\cdots(T_x+a_0x)f(x)=0.
\end{align}
When $M=N=0,a_0=b_0=q^\frac{1}{2}$, the $q$-difference equation \eqref{linear eq} coincides with the equation \eqref{eq: mu eq} satisfied by the $\mu$-function. 

First, we present a formal series solution of the following non-homogeneous $q$-difference equation: 
\begin{align}
\label{N=0 non-homogeneous eq}
&\left[\prod_{M\geq j\geq 0}^{\longrightarrow}(T_x+a_jx q^{j})\right]f(x)=1. 
\end{align}
Then, the following lemma holds: 
\begin{lem}
\label{lem: 1}
$(1)$\ 
 The image of the $q$-Borel transformation of the $q$-difference operator in \eqref{N=0 non-homogeneous eq} is 
\begin{align}
\label{eq: lem1-0}
\mathcal{B}_q\left(\prod_{M\geq j\geq0}^{\longrightarrow}(T_x+a_jxq^j)\right)
  &=
 \left[\prod_{j=k+1}^M\left(1+a_j\xi q^M\right)\right]T_\xi^{M-k} \mathcal{B}_q\left(\prod_{k\geq j\geq0}^{\longrightarrow}(T_x+a_jxq^j)\right),\quad (0\leq k\leq M). 
\end{align}
In particular, we have
\begin{align}
\label{eq: lem1-1}
\mathcal{B}_q\left(\prod_{M\geq j\geq0}^{\longrightarrow}(T_x+a_jxq^j)\right)
  =
  \left[\prod_{j=0}^M\left(1+a_j\xi q^{M}\right)\right]
  T_\xi^{M+1}\mathcal{B}_q
      =
  \sum_{j=0}^{M+1}e_jq^{jM}\xi^jT_\xi^{M+1}\mathcal{B}_q. 
\end{align}
$(2)$\ The $q$-difference operator in \eqref{N=0 non-homogeneous eq} is expanded as follows: 
\begin{align}
\label{eq: lem1-2}
\prod_{M\geq j\geq0}^{\longrightarrow}(T_x+a_jxq^j)
  =\sum_{j=0}^{M+1}e_jq^{-\frac{j(j-1)}{2}+jM}x^jT_x^{M+1-j}=\sum_{j=0}^{M+1}e_jq^\frac{j(j-1)}{2}T_x^{M+1-j}x^j, 
\end{align}
where 
\begin{align*}
e_j:=\sum_{0\leq k_1< \cdots < k_j\leq M}a_{k_1}\cdots a_{k_j}.
\end{align*}
$(3)$\ The $q$-difference equation \eqref{N=0 non-homogeneous eq} has the following formal series solution around $x=0$: 
\begin{align}
\label{eq: formal series}
f_0(x):=\sum_{n=0}^\infty (-1)^nh_nq^{-\frac{n(n+1)}{2}}x^n, 
\end{align}
where 
\begin{align}
\label{eq: generating func of comp}
h_j:=\sum_{0\leq k_1\leq \cdots \leq k_j\leq M}a_{k_1}\cdots a_{k_j}. 
\end{align}
\end{lem}
The $q$-difference operator in \eqref{N=0 non-homogeneous eq} is a $q$-commutative analogue of the generating function of the elementary symmetric polynomials: 
\begin{align}
\label{eq: generating func of e}
E(a_0,\ldots, a_M;x):=\prod_{j=0}^{M}(1+a_jx)=\sum_{j=0}^{M+1}e_jx^j, 
\end{align}
and the following relation
\begin{align}
&\left[\sum_{j=0}^{M+1}q^\frac{j(j-1)}{2}T_x^{M+1-j}e_jx^j\right]\sum_{n=0}^\infty (-1)^nq^{-\frac{n(n+1)}{2}}h_nx^n=1
\end{align}
is a $q$-commutative analogue of the fundamental relation (which we call the Wronski-type relation) between the elementary symmetric polynomials and the complete homogeneous symmetric polynomials: 
\begin{align}
\label{eq: sym wron}
E(a_0,\ldots, a_M;x)H(a_0,\ldots, a_M;-x)
&=\sum_{j=0}^{M+1}e_jx^j\sum_{n=0}^\infty (-1)^nh_n x^n=1, 
\end{align}
where the function $H(a_0,\ldots, a_M;x)$ is a generating function of the complete homogeneous symmetric polynomials: 
\begin{align*}
H(a_0,\ldots, a_M;x):=\frac{1}{E(a_0,\ldots, a_M;-x)}=\prod_{j=0}^M\frac{1}{1-a_jx}=\sum_{n=0}^\infty h_n x^n. 
\end{align*}

Furthermore, we prove that the function $\mathfrak{M}(x,a_0,\ldots,a_M)$ coincides with the image of $\mathcal{L}_q\circ\mathcal{B}_q$ of the formal series solution $f_0$ for the non-homogeneous $q$-difference equation \eqref{N=0 non-homogeneous eq}, and the image is essentially expressed by a sum of $\mu$-functions

\begin{thm}
\label{thm: 1}
The following identities hold: 
\begin{align}
\label{eq: image}
\mathcal{L}_q\circ\mathcal{B}_q\left(f_0\right)(x,-\lambda^{-1})
  &=
  -\mathfrak{M}\left(x\lambda,\frac{\lambda}{a_0},\ldots,\frac{\lambda}{a_M}\right), 
  \\
\label{eq: mu and M}
  \mathfrak{M}(x,a_0,\ldots, a_M)
  &=iq^\frac{1}{8}x^{-\frac{1}{2}}\sum_{j=0}^M a_j^\frac{1}{2}\left[\prod_{0\leq l\neq j\leq M}\frac{a_l}{a_l-a_j}\right]\mu(x,a_j). 
\end{align}
\end{thm}

Next, we present a system of fundamental solutions for the factorized $q$-difference equation \eqref{linear eq}. 
\begin{thm}
\label{mu eq}
The $q$-difference equation $(\ref{linear eq})$ has the following fundamental solutions:
\begin{align}
\label{basic solution}
  \frac{1}{\theta_q(-a_jx)}\quad (j=0,\ldots,M), 
  \quad 
  x^{\beta_k}
  \mathfrak{M}\left(x\lambda,\frac{b_k\lambda}{a_0},\ldots, \frac{b_k\lambda}{a_M}\right)\quad (k=0,\ldots,N), 
\end{align}
where $b_k=q^{\beta_k}$.
\end{thm}
From Theorem \ref{thm: 1} and \ref{mu eq}, we see that the $\mu$-function is a fundamental and important function in view of the factorized $q$-difference equations \eqref{linear eq}. 
This result shows that the Zwegers $\mu$-function is not merely a special function arising in mock modular forms, but is a natural building block of higher-order factorized $q$-difference equations.
Moreover, as a corollary of Theorem \ref{mu eq}, a special case of the one parameter deformation of the $\mu$-function $\mu(x,y;a;q)$ also appears as a solution of a higher order factorized $q$-difference equation. 
\begin{coro}
\label{coro: 1}
The factorized $q$-difference equation 
\begin{align*}
\left[\prod_{k=0}^N(T_x-b_k)\right](T_x+x)^{M+1}f(x)=0
\end{align*}
has the following fundamental solutions: 
\begin{align*}
\frac{1}{\theta_q(-xq^{-j})} \quad (j=0,\ldots,M),\quad 
x^{\beta_k}\mathfrak{M}\left(x\lambda,b_k\lambda,b_k\lambda q,\ldots,b_k\lambda q^M\right) \quad (k=0,\ldots,N). 
\end{align*}
Moreover, the function $\mathfrak{M}\left(x,b,bq,\ldots b q^M\right)$ essentially coincides with the generalized $\mu$-function \eqref{def: gmu}: 
\begin{align*}
\mu\left(x,b;q^{M+1};q\right)=iq^{-\frac{1}{8}}\left(\frac{x}{b}\right)^\frac{M+1}{2}\mathfrak{M}\left(x,b,bq,\ldots , b q^M\right). 
\end{align*}
\end{coro}
 
Furthermore, we show that the higher level Appell functions \eqref{eq: higher} satisfy special cases of the factorized $q$-difference equations \eqref{linear eq}. 
 \begin{thm}
 \label{Appell eq}
We assume $G_m(x,y):=x^{-\frac{m}{2}}A_m(x,(-1)^{m-1}y)$. 
Then, the function $G_m$ satisfies the following pseudo-periodicity and the following $(m+1)$-th order factorized $q$-difference equation: 
\begin{align}
\label{Gm quasi-perio}
yG_m(xq,y)+x^mG_m(x,y)+\sum_{k=0}^{m-1}x^k\theta_{q^m}(-yq^k)&=0, 
\\
\label{A}
\left[\prod_{k=1}^{m}(T_x-q^{k-1})\right](T_x+x^my^{-1})G_m(x,y)&=0. 
\end{align}
\end{thm}
Since the equation \eqref{A} is the special case of $M=0$, $N=m-1$, $a_0=y^{-1}$ and $b_k=q^k$ of \eqref{linear eq} after the transformations $x$ into $x^m$ and $q$ into $q^m$, using Theorem $\ref{mu eq}$ and Theorem $\ref{Appell eq}$, we show that the higher level Appell functions are expressed in terms of the $\mu$-function, and give an alternative proof of (\ref{Z}). 

Next, we give some non-homogeneous $q$-difference relations, connection formulas and symmetric properties of the function $\mathfrak{M}(x,a_0,\ldots, a_M)$. 

\begin{prop}
\label{prop: 1}
$\mathfrak{M}(x,a_0,\ldots,a_M)$ satisfies the following non-homogeneous $q$-difference equations: 
\begin{align}
\label{eq: prop1-1}
\left[T_x+\frac{x}{a_0}\right]\mathfrak{M}_{M+1}(x,a_0,\ldots,a_M)&=\mathfrak{M}_M\left(x,\frac{a_1}{q},\ldots,\frac{a_M}{q}\right), 
\\
\label{eq: prop1-2}
\left[\prod_{M\geq j\geq 0}^{\longrightarrow}\left(T_x+\frac{xq^j}{a_j}\right)\right]\mathfrak{M}(x,a_0,\ldots,a_M)&=-1. 
\end{align}
Moreover, the following identities hold: 
\begin{align}
\label{eq: prop1-3}
\mathfrak{M}(x,a_0,\ldots, a_M)
  &=
  \mathfrak{M}(xz,a_0z,\ldots,a_Mz)+\frac{\theta_q(-z)(q)_\infty^3}{\theta_q(-x)\theta_q(-xz)}
\sum_{j=0}^M \frac{a_j\theta_q(-a_jxz)}{\theta_q(-a_j)\theta_q(-a_jz)}\left[\prod_{l\neq j}\frac{a_l}{a_l-a_j}\right], 
  \\
\label{eq: prop1-4}
  \mathfrak{M}(x,a_0,\ldots, a_M)
  &=
\mathfrak{M}(xq,a_0q,\ldots, a_Mq)
  =
\mathfrak{M}\left(x,a_{\sigma(0)},\ldots, a_{\sigma(M)}\right), 
\end{align}
where $\sigma\in\mathfrak{S}_{M+1}$. 
\end{prop}

The above results are for general complex parameters with $a_j\ (0\leq j\leq M)$. Here, by restricting the parameters $a_j$ to $a q^{k_j}$ for $k_j\in\Zz$, we obtain the following symmetry in $a$ and $x$, in adding to the above properties. 

\begin{coro}
\label{coro: 2}
For $k_0,\ldots, k_M\in\Zz$, $\prod_{m<n}(k_m-k_n)\neq0$, the following identities hold: 
\begin{align}
\label{eq: coro2-1}
\mathfrak{M}\left(x,a q^{k_0},\ldots,a q^{k_M}\right)
  &=
  \mathfrak{M}\left(xz,az q^{k_0},\ldots,az q^{k_M}\right)
    \nonumber\\
    &\qquad+\frac{a\theta_q(-z)\theta_q(-axz)(q)_\infty^3}{\theta_q(-x)\theta_q(-xz)\theta_q(-a)\theta_q(-az)}
\sum_{j=0}^M \left[\prod_{l\neq j}\frac{1}{1-q^{k_j-k_l}}\right]\left(-\frac{a}{x}\right)^{k_j}q^\frac{k_j(k_j+1)}{2}, 
  \\
\label{eq: coro2-2}
  \mathfrak{M}\left(x,a q^{k_0},\ldots,a q^{k_M}\right)
  &=
\mathfrak{M}\left(\frac{1}{a},\frac{q^{k_0}}{x},\ldots,\frac{q^{k_M}}{x}\right). 
\end{align}
\end{coro}

Furthermore, we show that the properly normalized function $\mathfrak{M}\left(x,a q^{k_0},\ldots,a q^{k_M}\right)$ is expressed in terms of the Zwegers' $\mu$-function, and we present a modular completion of it. 

\begin{prop}
\label{prop: 2}
For ${\bf{k}}=(k_0,\ldots, k_M)\in\Zz^{M+1}$, $k_0<\cdots < k_M$, ${\rm Im}(\tau)>0$, $q=e^{2\pi i\tau}$, $u,v\not\in\Zz\tau+\Zz$, we define
\begin{align*}
J_{\bf{k}}(u)
  &:=
  \sum_{j=0}^M\left[\prod_{l\neq j}\frac{1}{1-q^{k_j-k_l}}\right]e^{2\pi ik_j\left(u+\frac{1}{2}\right)}q^\frac{k_j(k_j+1)}{2}, 
\\
H_{\bf{k}}(u)
  &:=
  \sum_{j=0}^M\left[\prod_{l\neq j}\frac{1}{1-q^{k_j-k_l}}\right]\sum_{n=-|k_j|}^{|k_j|-1}\varepsilon(k_j,n)e^{2\pi i(k_j-n)\left(u+\frac{1}{2}\right)}q^{\frac{k_j(k_j+1)}{2}-\frac{n(n+1)}{2}}, 
  \\
\mathfrak{m}_{\bf{k}}(u,v;\tau)
  &:=
  e^{\pi i(u-v)}q^{-\frac{1}{8}}\frac{\mathfrak{M}\left(e^{2\pi iu},e^{2\pi iv}q^{k_0},\ldots, e^{2\pi iv}q^{k_M}\right)}{J_{\bf{k}}(v-u)}. 
\end{align*}
Then, we obtain
\begin{align}
\label{eq: prop2-1}
\mathfrak{m}_{\bf{k}}(u,v)
  &=
  -i\mu\left(u,v\right)+e^{\pi i(u-v)}q^{-\frac{1}{8}}\frac{H_{\bf{k}}(v-u)}{J_{\bf{k}}(v-u)}, 
  \\
\label{eq: prop2-2}
\mathfrak{m}_{\bf{k}}(u,v)
  &=
  \mathfrak{m}_{\bf{k}}(u+w,v+w)-\frac{\vartheta(w)\vartheta(u+v+w)\eta(\tau)^3}{\vartheta(u)\vartheta(v)\vartheta(u+w)\vartheta(v+w)}, 
  \\
\label{eq: prop2-3}
-i\widetilde{\mu}(u,v;\tau)
  &=
  \mathfrak{m}_{\bf{k}}(u,v)+\frac{1}{2}R(u-v)-e^{\pi i(u-v)}q^{-\frac{1}{8}}\frac{H_{\bf{k}}(v-u)}{J_{\bf{k}}(v-u)}, 
\end{align}
where
\begin{align*}
\varepsilon(x,y):=\frac{\sgn\left(x+\frac{1}{2}\right)+\sgn\left(y+\frac{1}{2}\right)}{2},\quad \sgn(x)&:=\begin{cases} 1 & x\geq 0  \\ -1 & x<0\end{cases}, 
\end{align*}
and the definition of $R(u;\tau)$ and $\widetilde{\mu}(u,v;\tau)$ are given by \eqref{eq: defi R} and \eqref{eq: comp mu}.
\end{prop}

This paper is organized as follows. First, in Section \ref{sec: 2}, we prove the main results Theorem \ref{mu eq} and Theorem \ref{Appell eq} and their corollaries, and as an application, we give another proof of (\ref{Z}) from the viewpoint of $q$-difference equations. 
In particular, for the proof of Theorem \ref{mu eq}, we show that the function $\mathfrak{M}(x,a_0,\ldots,a_M)$ is introduced from the $q$-Borel summation method of a divergent solution of \eqref{linear eq}. 
Next, in Section \ref{sec: 3}, we prove Proposition \ref{prop: 1} and \ref{prop: 2}. 
In particular, the formula \eqref{eq: prop1-1} of Proposition \ref{prop: 1} was shown by G. Shibukawa, and we prove it according to the proof. 
Finally, we present fundamental solutions of another factorized $q$-difference equation and a factorized differential equation in Appendices A and B, respectively. 

\section{Proof of Theorem \ref{thm: 1}, \ref{mu eq} and \ref{Appell eq}}\label{sec: 2}
In this section, we prove Theorem \ref{thm: 1}, \ref{mu eq} and \ref{Appell eq} and their corollaries. 

\begin{proof}[{\bf Proof of Lemma $\ref{lem: 1}$}]
(1)\ The equality \eqref{eq: lem1-0} is obtained by iterating the relation \eqref{eq: q-Borel trans} and using the relation
\begin{align}
\label{eq: exchange row}
T_\xi \xi=q\xi T_\xi. 
\end{align}
The second equality of \eqref{eq: lem1-1} follows from \eqref{eq: generating func of e}. \\
(2)\ 
The first equality \eqref{eq: lem1-2} is obtained by applying $\mathcal{L}_q$ to the last expansion of \eqref{eq: lem1-1} from \eqref{eq: q-Laplace trans}. 
Moreover, the second equality is obtained from \eqref{eq: exchange row}. \\
(3)\ 
 Let the formal series $\sum_{n\geq0}A_n x^n$ be a solution of the non-homogeneous $q$-difference equation \eqref{N=0 non-homogeneous eq}. 
Then, from \eqref{eq: lem1-2} of Lemma \ref{lem: 1}, we have 
\begin{align*}
\left[\prod_{M\geq j\geq 0}^{\longrightarrow}(T_x+a_jx q^{j})\right]
  \sum_{n=0}^\infty 
  A_n x^{n}
  &=
    \sum_{n=0}^\infty x^{n}
    \left(\sum_{j=0}^{M+1} A_{n-j}q^{(M+1-j)n+\frac{j(j-1)}{2}}e_j\right), 
\end{align*}
where $A_0=1, A_{-1}=\cdots =A_{-M-1}=0$.
Then, putting $A_n=B_nq^{-\frac{n(n+1)}{2}}$, we obtain
\begin{align*}
&\left[\prod_{M\geq j\geq 0}^{\longrightarrow}(T_x+a_jx q^{j})\right]\sum_{n\geq 0}B_nq^{-\frac{n(n+1)}{2}}x^n
    =\sum_{n=0}^\infty\left(\sum_{j=0}^{M+1}B_{n-j}e_j\right)q^{-\frac{n(n+1)}{2}}\left(x q^{M+1}\right)^n=1,
\end{align*}
and the coefficients $B_n$ satisfy the following relations: 
\begin{align*}
&\sum_{j=0}^{M+1}B_{n-j}e_j=0\quad (n\geq1),\quad B_0=1, B_{-1}=\cdots =B_{-M-1}=0. 
\end{align*}
Therefore, from the relation \eqref{eq: sym wron}, we obtain $B_n=(-1)^nh_n$ and the formal series solution \eqref{eq: formal series}.  
\end{proof}
\begin{proof}[{\bf Proof of Theorem $\ref{thm: 1}$}]
Using the generating function of the complete homogeneous symmetric polynomials \eqref{eq: generating func of comp}, the image of the $q$-Borel transformation of \eqref{eq: formal series} is
\begin{align*}
\mathcal{B}_q\left(f_0\right)(\xi)
    &=\sum_{n=0}^\infty (-1)^nh_n\left(\frac{\xi}{q}\right)^n
    =\prod_{j=0}^{M}\frac{1}{1+(a_j\xi/q)}, 
\end{align*}
and the image of $\mathcal{L}_q\circ\mathcal{B}_q$ is 
\begin{align}
 \label{fra}
\mathcal{L}_q\left(\prod_{j=0}^{M}\frac{1}{1+(a_j\xi/q)}\right)(x,\lambda)
    =\sum_{n\in\Zz}\frac{1}{\theta_q(\lambda q^n/x)}\prod_{j=0}^M\frac{1}{1+a_j\lambda q^{n-1}}. 
\end{align}
Therefore, from the following identities
\begin{align}
\label{theta trans}
x^nq^\frac{n(n-1)}{2}\theta_q(xq^n)=\theta_q(x)\quad (n\in\Zz), 
\end{align}
we obtain \eqref{eq: image}. 

Moreover, from a partial fraction expansion
\begin{align*}
\prod_{j=0}^M\frac{1}{1-a_j\lambda}=\sum_{j=0}^M\left(\prod_{l\neq j}\frac{1}{1-a_l/a_j}\right)\frac{1}{1-a_j\lambda}, 
\end{align*}
we have 
\begin{align*}
\mathfrak{M}(x,a_0,\ldots,a_M)
  &=
  \sum_{j=0}^M
  \left[\prod_{0\leq l\neq j\leq M}\frac{a_l}{a_l-a_j}\right]
  \frac{1}{\theta_q(-x)}
    \sum_{n\in\Zz}
    \frac{(-1/x)^nq^\frac{n(n+1)}{2}}{1-(q^n/a_j)}
  \\
  &=
  -\sum_{j=0}^M
  \left[\prod_{0\leq l\neq j\leq M}\frac{a_l}{a_l-a_j}\right]
  \frac{a_j}{\theta_q(-x)}\sum_{n\in\Zz}\frac{(-x)^nq^\frac{n(n+1)}{2}}{1-a_jq^n}. 
\end{align*}
Therefore, from the symmetry of the $\mu$-function $\mu(x,y)=\mu(y,x)$ (see \cite{Zw1}), we obtain \eqref{eq: mu and M}. 
\end{proof}

\begin{proof}[{\bf Proof of Theorem $\ref{mu eq}$}]
From the following exchange law of $q$-difference operators: 
\begin{align*}
(T_x+ax)(T_x+bx)=(T_x+bx/q)(T_x+axq), 
\end{align*}
it is clear that the theta functions $1/\theta_q(-a_jx)$ are special solutions of the $q$-difference equation \eqref{linear eq}. 
Moreover, the function $x^{\beta_k}\mathfrak{M}\left(x\lambda,\frac{b_k\lambda}{a_0},\ldots, \frac{b_k\lambda}{a_M}\right)$is a solution of the following factorized $q$-difference equation: 
\begin{align*}
(T_x-b_k)\left[\prod_{M\geq j\geq0}^{\longrightarrow}(T_x+a_jxq^j)\right]f(x)=0, 
\end{align*}
from Lemma \ref{lem: 1} and Theorem \ref{thm: 1}. 
Therefore, we obtain the required fundamental solutions of \eqref{linear eq}. 
\end{proof}
\begin{rmk}
We provide another derivation of the solution \eqref{fra} of the $q$-difference equation \eqref{N=0 non-homogeneous eq} without using the $q$-Borel summation method. 
This derivation method was discovered together with Prof. Y. Yamada. 
First, by putting $f(x)=\frac{g(x)}{\theta_q(-x/\lambda)}$, we rewrite the $q$-difference equation  \eqref{N=0 non-homogeneous eq} as 
\begin{align*}
\left[\prod_{j=0}^M(T_x-a_j\lambda)\right]g(x)=\theta_q(-xq^{M+1}/\lambda). 
\end{align*}
Hence, putting $g(x)=\sum_{n\in\Zz}A_nx^n$ and comparing the coefficients of $x^n$, we obtain
\begin{align*}
A_n=\frac{(-q^{M+1}/\lambda)^nq^\frac{n(n-1)}{2}}{(q^n-a_M\lambda)\cdots(q^n-a_0\lambda)}. 
\end{align*} 
Therefore, the function 
\begin{align*}
f(x)
  &=
  \frac{1}{\theta_q(-x/\lambda)}
  \sum_{n\in\Zz}
  \frac{(-x q^{M+1}/\lambda)^nq^\frac{n(n-1)}{2}}{(q^n-a_M\lambda)\cdots(q^n-a_0\lambda)}
  \\
  &=
  \frac{1}{\theta_q(-x/\lambda)}
  \sum_{n\in\Zz}
  \frac{(-\lambda/x)^nq^\frac{n(n+1)}{2}}{(1-a_M\lambda q^n)\cdots(1-a_0\lambda q^n)}
\end{align*}
is a solution of the $q$-difference equation  \eqref{N=0 non-homogeneous eq} and we derive the required solution without using the $q$-Borel and $q$-Laplace transformations. 
\end{rmk}
\begin{proof}[{\bf Proof of Corollary \ref{coro: 1}}]
Putting $a_j=q^{-j}$, the function \eqref{fra} is expressed by the generalized $\mu$-function \eqref{def: gmu}: 
\begin{align*}
\mathcal{L}_q\circ\mathcal{B}_q\left(\sum_{n=0}^\infty A_n x^n\right)(x,-\lambda^{-1})=iq^\frac{1}{8}x^{-\frac{M+1}{2}}\mu(1/\lambda,1/x\lambda;q^{M+1};q). 
\end{align*}
From the symmetry of the generalized $\mu$-function (see \cite{ST}): 
\begin{align*}
\mu(x,y;q^{M+1};q)=\mu(x^{-1},y^{-1};q^{M+1};q)=\mu(y,x;q^{M+1};q), 
\end{align*} 
we get the conclusion. 
\end{proof}
\begin{coro}
\label{non-homogeneous eq}
We assume that a function $f(x)$ satisfies the following first order non-homogeneous $q$-difference equation: 
\begin{align}
\label{tukareta}
(T_x+x)f(x)=\sum_{k=0}^{M}P_k(x),
\end{align}
where $P_k(x)$ is a pseudo-periodic function $P_k(xq)=b_kP_k(x)$. 
Then the function $f(x)$ is expressed as a linear combination of the functions 
\begin{align*}
\frac{1}{\theta_q(-x)},\quad x^{\beta_k-\frac{1}{2}}\mu(x\lambda,b_k\lambda)\quad (k=0,\ldots,M).
\end{align*}
\end{coro}
\begin{proof}[{\bf Proof}]
From the equation (\ref{tukareta}), we have 
\begin{align*}
\left[\prod_{k=0}^{M}(T_x-b_k)\right](T_x+x)f(x)=\left[\prod_{k=0}^{M}(T_x-b_k)\right]\sum_{k=0}^{M}P_k(x)=0. 
\end{align*}
Hence the function $f(x)$ is expressed as  a linear combination of fundamental solutions (\ref{basic solution}), since $f(x)$ satisfies the $q$-difference equation (\ref{linear eq}). 
\end{proof}

Next, we prove Theorem $\ref{Appell eq}$ and its corollaries. 
\begin{proof}[{\bf Proof of Theorem $\ref{Appell eq}$}]
The pseudo-periodicity \eqref{Gm quasi-perio} is obtained by the following calculations: 
\begin{align*}
-yG_m(xq,y)
   &=\sum_{n\in\Zz}\frac{(-1)^{n}y^{n}q^\frac{mn(n-1)}{2}}{1-xq^n}(x^mq^{mn}+1-x^mq^{mn})\\
   &=x^mG_m(x,y)+\sum_{n\in\Zz}(-y)^{n}q^\frac{mn(n-1)}{2}\sum_{k=0}^{m-1}x^kq^{kn}\\
   &=x^mG_m(x,y)+\sum_{k=0}^{m-1}x^k\theta_{q^m}(-yq^k). 
\end{align*}
Moreover, by applying the operator $\prod_{k=1}^{m}(T_x-q^{k-1})$ to \eqref{Gm quasi-perio}, we immediately obtain the $q$-difference equation \eqref{A}. 
\end{proof}
\begin{proof}[{\bf Another Proof of the equation $(\ref{Z})$}]
From Corollary \ref{non-homogeneous eq}, $G_m(x,y)$ is expressed by a linear combination of 
\begin{align*}
\frac{1}{\theta_{q^m}(-x^m/y)},\quad x^{k-\frac{m}{2}}\mu(x^m\lambda,y\lambda q^k;q^m)\quad (k=0,\ldots,m-1). 
\end{align*}
We put
\begin{align*}
G_m(x,y)=\frac{C(x,\lambda)}{\theta_{q^m}(-x^m/y)}+\sum_{k=0}^{m-1}C_k(x,\lambda)x^{k-\frac{m}{2}}\mu(x^m\lambda,y\lambda q^k;q^m), 
\end{align*}
where $C(x,\lambda)=C(x q,\lambda)$ and $C_k(x,\lambda)=C_k(xq,\lambda)$. 
From the pseudo-periodicity of theta function (\ref{theta trans}), the $\mu$-function (cf. \cite[Proposition 1.4]{Zw1})
\begin{align}
\label{eq: pseudo periodicity of mu}
\mu(xq,y)=-\frac{x}{y}q^\frac{1}{2}\mu(x,y)-i\sqrt{\frac{x}{y}}q^\frac{3}{8}
\end{align}
and the equation (\ref{Gm quasi-perio}), we have 
\begin{align*}
yG_m(xq,y)+x^mG_m(x,y)=-i\sqrt{y}q^{-\frac{m}{8}}\sum_{k=0}^{m-1}C_k(x,\lambda)x^kq^\frac{k}{2}=-\sum_{k=0}^{m-1}x^k\theta_{q^m}(-yq^k). 
\end{align*}
Hence, we have
\begin{align*}
C_k(x,\lambda)=-iq^\frac{m}{8}\frac{\theta_{q^m}(-yq^k)}{\sqrt{yq^k}}, 
\end{align*}
and $C(x,\lambda)$ is rewritten as 
\begin{align}
\label{eq: C(x)}
C(x,\lambda)=\theta_{q^m}(-x^m/y)\left\{G_m(x,y)+iq^\frac{m}{8}\sum_{k=0}^{m-1}\frac{\theta_{q^m}(-yq^k)}{\sqrt{yq^k}}x^{k-\frac{m}{2}}\mu(x^m\lambda,y\lambda q^k;q^m)\right\}. 
\end{align}
Since the right-hand side of \eqref{eq: C(x)} is an elliptic function with a simple pole at $x=1$ with $\lambda=1$, $C(x,1)$ is a constant function. Therefore, we have 
\begin{align*}
C(x,1)=&\lim_{x\to {}^m\sqrt{y}}\theta_{q^m}(-x^m/y)\left\{G_m(x,y)+iq^\frac{m}{8}\sum_{k=0}^{m-1}\frac{\theta_{q^m}(-yq^k)}{\sqrt{yq^k}}x^{k-\frac{m}{2}}\mu(x^m,yq^k;q^m)\right\}\\
   &=0. 
\end{align*}
and we obtain \eqref{Z}. 
\end{proof}

Since the order of the equation \eqref{A} degenerates at $y=1$, the $q$-difference equation for $G_m(x,1)$ and its relations to the fundamental solutions are summarized in the following corollary: 
\begin{coro}\ \\
\label{Gm coro}
The function $G_m(x,1)$ satisfies the following pseudo-periodicity and the following $m$-th order factorized $q$-difference equation: 
\begin{align}
\label{eq: Gm quasi perio}
G_m(xq,1)+x^mG_m(x,1)+\sum_{k=1}^{m-1}x^k\theta_{q^m}(-q^k)&=0, 
  \\
  \label{Gm1 eq}
\left[\prod_{k=1}^{m-1}(T_x-q^k)\right](T_x+x^m)G_m(x,1)&=0.
\end{align}
Moreover, the above $q$-difference equation $(\ref{Gm1 eq})$ has the following fundamental solution: 
\begin{align}
\label{eq: fundamental sol}
\frac{1}{\theta_{q^m}(-x^m)},\quad x^{j-\frac{m}{2}}\mu(x^m\lambda,\lambda q^j;q^m)\quad(j=1,\ldots,m-1), 
\end{align}
and the function $G_m(x,1)$ is expressed as follows: 
\begin{align}
\label{Gm mu trans}
G_m(x,1)&=-iq^\frac{m}{8}\sum_{j=1}^{m-1}\theta_{q^m}(-q^j)x^{j-\frac{m}{2}}q^{-\frac{j}{2}}\mu(x^m\lambda,\lambda q^j;q^m)
  +
   \frac{\theta_{q^m}(-\lambda)\theta_q(-x\lambda)(q)_\infty^3}{\theta_{q^m}(-x^m\lambda)\theta_q(-\lambda)\theta_q(-x)}. 
\end{align}
\end{coro}

\begin{proof}[{\bf Proof}]
\ \\
First, we immediately obtain \eqref{eq: Gm quasi perio} by substituting $y=1$ in \eqref{Gm quasi-perio}, \eqref{Gm1 eq} by applying $\prod_{k=1}^{m-1}(T_x-q^k)$ to \eqref{eq: Gm quasi perio}, and \eqref{eq: fundamental sol} from Corollary \ref{non-homogeneous eq}. 

Furthermore, since
\begin{align*}
\lim_{y\to 1}
\theta_{q}(-y)\mu(x,y;q)=\frac{(q)_\infty^3}{\theta_q(-x)}, 
\end{align*}
we have 
\begin{align*}
G_m(x,1)=\frac{(q^m;q^m)_\infty^3}{\theta_{q^m}(-x^m)}-\sum_{j=1}^{m-1}i\theta_{q^m}(-q^j)x^{j-\frac{m}{2}}q^{\frac{m}{8}-\frac{j}{2}}\mu(x^m,q^j;q^m). 
\end{align*}
Then, from the translation formula of the $\mu$-function (cf. \cite[Proposition 1.4]{Zw1})
\begin{align*}
iq^\frac{1}{8}\mu(xz,yz; q)=iq^\frac{1}{8}\mu(x,y; q)+\frac{\sqrt{xy}(q)_\infty^3\theta_q(-z)\theta_q(-xyz)}{\theta_q(-x)\theta_q(-y)\theta_q(-xz)\theta_q(-yz)}, 
\end{align*}
we obtain 
\begin{align*}
G_m(x,1)&=-iq^\frac{m}{8}\sum_{j=1}^{m-1}\theta_{q^m}(-q^j)x^{j-\frac{m}{2}}q^{-\frac{j}{2}}\mu(x^m\lambda,\lambda q^j;q^m)\\
   &\qquad+\frac{(q^m;q^m)_\infty^3}{\theta_{q^m}(-x^m)}\sum_{j=0}^{m-1}\frac{x^j\theta_{q^m}(-\lambda)\theta_{q^m}(-x^m\lambda q^j)}{\theta_{q^m}(-x^m\lambda)\theta_{q^m}(-\lambda q^j)}. 
 \end{align*}
Therefore, since
\begin{align*}
\sum_{j=0}^{m-1}\frac{x^j\theta_{q^m}(-x^m\lambda q^j)}{\theta_{q^m}(-\lambda q^j)}=\frac{\theta_q(-x\lambda)\theta_{q^m}(-x^m)(q)_\infty^3}{\theta_q(-\lambda)\theta_q(-x)(q^m;q^m)_\infty^3}, 
\end{align*}
we obtain \eqref{Gm mu trans}. 
\end{proof}

\section{Proof of Proposition \ref{prop: 1} and \ref{prop: 2}}
\label{sec: 3}
In this section, we prove Proposition \ref{prop: 1} and \ref{prop: 2}. 
In particular, for the formula \eqref{eq: prop1-1} of Proposition \ref{prop: 1}, we present a proof by following G. Shibukawa. 
\begin{proof}[{\bf Proof of Proposition $\ref{prop: 1}$}]
From the definition of $\mathfrak{M}(x,a_0,\ldots, a_M)$, we have
\begin{align*}
\left[T_x+\frac{x}{a_0}\right]\mathfrak{M}_{M+1}(x,a_0,\ldots, a_M)
  &=
  -\frac{1}{\theta_q(-x)}
   \sum_{n\in\Zz}\left(-x\left(-\frac{1}{xq}\right)^n+\left(-\frac{1}{x}\right)^n\frac{x}{a_0}\right)q^\frac{n(n+1)}{2}\prod_{j=0}^M\frac{1}{1-(q^n/a_j)}
   \\
   &=
   -\frac{1}{\theta_q(-x)}\sum_{n\in\Zz}\left(-\frac{1}{x}\right)^{n-1}q^\frac{n(n-1)}{2}
   \prod_{j=1}^M\frac{1}{1-(q^n/a_j)}. 
\end{align*} 
Therefore, by replacing $n$ by $n+1$, we obtain \eqref{eq: prop1-1}. 
Further, by iteratively applying \eqref{eq: prop1-1}, \eqref{eq: prop1-2} holds. 

Moreover, from the transformation of the $\mu$-function \cite{Zw1}: 
\begin{align*}
\mu(x,y)=\mu(xz,yz)-\frac{i q^\frac{1}{8}\sqrt{xy}\theta_q(-z)\theta_q(-xyz)(q)_\infty^3}{\theta_q(-x)\theta_q(-y)\theta_q(-xz)\theta_q(-yz)}, 
\end{align*}
we obtain \eqref{eq: prop1-3}. 
Finally, the equation \eqref{eq: prop1-4} is clear from the symmetry of the $\mu$-function \cite{Zw1} $\mu(x,y)=\mu(xq,yq)$ and the definition of $\mathfrak{M}(x,a_0,\ldots, a_M)$. 
\end{proof}
\begin{proof}[{\bf Proof of Corollary $\ref{coro: 2}$}]
Putting $a_j=a q^{k_j}$ in \eqref{eq: prop1-3}, we immediately obtain \eqref{eq: coro2-1} from the pseudo-periodicity of the theta function \eqref{theta trans}. 
Moreover, by putting $z=1/ax$ in \eqref{eq: coro2-1}, we obtain \eqref{eq: coro2-2}. 
\end{proof}
\begin{proof}[{\bf Proof of Proposition $\ref{prop: 2}$}]
From the pseudo-periodicity of the $\mu$-function \eqref{eq: pseudo periodicity of mu}, we have 
\begin{align*}
\mathfrak{M}\left(x,a q^{k_0},\ldots, a q^{k_M}\right)
  &=
  \frac{a}{\theta_q(-a)}\sum_{n\in\Zz}\frac{(-a)^nq^\frac{n(n+1)}{2}}{1-x q^n}\sum_{j=0}^M\left(-\frac{a}{x}\right)^{k_j}q^{\frac{k_j(k_j+1)}{2}}
  \\
  &\quad +\sum_{j=0}^M\left[\prod_{l\neq j}\frac{1}{1-q^{k_j-k_l}}\right]
    \sum_{m=-|k_j|}^{|k_j|-1}
    \varepsilon(k_j,m)\left(-\frac{a}{x}\right)^{k_j-m}q^{\frac{k_j(k_j+1)}{2}-\frac{m(m+1)}{2}}. 
\end{align*}
Therefore, the function $\mathfrak{M}\left(x,a q^{k_0},\ldots, a q^{k_M}\right)$ is expressed by the $\mu$-function as 
\begin{align*}
e^{\pi i(u-v)}q^{-\frac{1}{8}}\mathfrak{M}\left(e^{2\pi iu},e^{2\pi iv}q^{k_0},\ldots, e^{2\pi iv}q^{k_M}\right)=-i\mu(u,v)J_{\bf{k}}(v-u)+e^{\pi i(u-v)}q^{-\frac{1}{8}}H_{\bf{k}}(v-u),
\end{align*}
and we obtain \eqref{eq: prop2-1}. 
Furthermore, from \eqref{eq: coro2-1}, \eqref{eq: prop2-1} and the definition of $\mathfrak{m}_{\bf{k}}$ and $\widetilde{\mu}$, we obtain \eqref{eq: prop2-2} and \eqref{eq: prop2-3}. 
\end{proof}

\section*{Appendix A}

In this Appendix, we present a fundamental solution of the following factorized $q$-difference equation: 
\begin{align}
\label{eq: factorized equation minus} 
\left[\prod_{k=0}^N(1-b_kT_x)\right]\left[\prod_{j=0}^{M}(1+a_jxT_x)\right]f(x)=0. 
\end{align}
The $q$-difference equation \eqref{eq: factorized equation minus} is derived by the transformations $q$ into $q^{-1}$, $a_j$ into $a_j/q$ and $x$ into $x q^{M+N+2}$ of \eqref{linear eq}. 
Then, the $q$-difference equation has the following fundamental solution around $x=0$: 
\begin{align*}
\theta_q(-a_jx), \quad \frac{\theta_q(b_kx)}{\theta_q(x)}\mathcal{B}_q(H(a_0/b_k,\ldots, a_M/b_k;-x)) \quad (0\leq j\leq M,0\leq k\leq N). 
\end{align*}
In particular, if we put $N=0$ and $b_0=1$, then the following series
\begin{align}
\label{eq: false jacobi}
&2\mathcal{B}_q(H(a_0,\ldots, a_M;-x))-\sum_{j=0}^M
  \prod_{0\leq l\neq j\leq M}\frac{a_j}{a_j-a_l}
    \theta_q(-a_jx)\nonumber\\
&=
\sum_{j=0}^M
  \prod_{0\leq l\neq j\leq M}\frac{a_j}{a_j-a_l}
    \sum_{n\in\Zz}\sgn\left(n+\frac{1}{2}\right)(-a_jx)^nq^\frac{n(n-1)}{2}
\end{align}
 is a special solution of \eqref{eq: factorized equation minus}. 
 Further, if we put $M=0$, then the series of the right hand side \eqref{eq: false jacobi} corresponds to the false theta function of $\theta_q(-a_jx)$ \cite{Ra}. 

\section*{Appendix B}
In this appendix, we present a fundamental solution of the following differential equation, which is a differential analogue of \eqref{linear eq}: 
\begin{align}
\label{eq: factorized differential equation}
\left[\prod_{k = 0}^{N}(x\partial_{x} + b_{k})\right]
   \left[\prod_{j = 0}^{M}
   (1 - a_{j}\partial_{x})\right]f(x)=0. 
\end{align}
The differential equation \eqref{eq: factorized differential equation} has the following formal series solutions: 
\begin{align*}
e^{\frac{x}{a_{j}}},\quad \mathfrak{m}_{M + 1}(x,a_{0}, \ldots, a_{M};b_{k}) \quad (0 \leq j \leq M, 0 \leq k \leq N), 
\end{align*}
where 
\begin{align*}
\mathfrak{m}_{M + 1}(x,a_{0}, \ldots, a_{M};b_{k})
   &:=
   \frac{1}{\Gamma{(b_{k})}}
   \int_{0}^{\infty }
      \frac{u^{b_{k}}e^{-xu}}{(1 + a_{0}u) \cdots (1 + a_{M}u)}\frac{du}{u}
   \\
  &=
  \frac{1}{\Gamma{(b_{k})}}
    \sum_{j=0}^M
      \prod_{l\neq j}\frac{a_j}{a_j-a_l}\int_0^\infty \frac{u^{b_k}e^{-x u}}{1+a_ju}\frac{du}{u}, 
\end{align*}
and the real part of the parameters $a_0,\ldots, a_M, b_0,\ldots, b_N,x$ are the real positive numbers. 
The function $\mathfrak{m}_{M + 1}(x,a_{0}, \ldots, a_{M};b_{k})$ is regarded as a differential analogue of $\mathfrak{M}_{M+1}(x,a_{0},\ldots, a_{M})$, and is formally expressed as follows: 
\begin{align*}
\mathfrak{m}_{M + 1}(x,a_{0}, \ldots, a_{M};b_{k})
  &=
\sum_{n \geq 0}
   h_{n}(a_{0},\ldots,a_{M})
   \frac{1}{\Gamma{(b_{k})}}
   \int_{0}^{\infty }
      u^{b_{k} + n}e^{-xu}\frac{du}{u}
  \\
   &=
   \sum_{n \geq 0}
   h_{n}(a_{0},\ldots,a_{M})
   (b_{k})_{n}
   x^{- b_{k} - n}
   \\
   &=
   x^{- b_{k}}
   \sum_{n \geq 0}
   h_{n}(a_{0},\ldots,a_{M})
   (b_{k})_{n}
   x^{- n}.    
\end{align*}

\section*{Acknowledgments}
The author is grateful to Professor Yasuhiko Yamada (Kobe University) for his helpful comments with respect to $q$-difference equations. 

\medskip
\begin{flushleft}
Satoshi Tsuchimi\\
Department of Mathematics\\
Graduate School of Science\\ 
Kindai University\\
3-4-1, Kowakae, Higashi-Osaka, Osaka, 577-8502\\ 
JAPAN\\
E-mail: tsuchimi@math.kindai.ac.jp
\end{flushleft}

\end{document}